\documentclass[10pt]{amsart}
\usepackage{amssymb,amsthm,amsmath,amsfonts}
\usepackage{hyperref}
\usepackage{color}

\def\R{{\mathbb R}}

\def\C{{\mathbb C}}

\def\Z{{\mathbb Z}}

\def\<{\langle}
\def\>{\rangle}

\def\E{{\mathbb E}}
\def\U{{\mathbb U}}
\def\V{{\mathbb V}}

\def\I{\mathbb I}

\def\X{{\mathbb X}}

\def\c{\underline c}

\def\v{\underline v}

\def\Y{{\mathbb Y}}
\def\Z{{\mathbb Z}}

\def\A{\mathcal{A}}
\def\B{\mathcal{B}}

\newtheorem{th-def}{Theorem-Definition}[section]
\newtheorem{theo}{Theorem}[section]
\newtheorem{lemm}[theo]{Lemma}

\newtheorem{Rem}[theo]{Remark}

\title{Convolution, subordination and characterization problems in noncommutative probability}
\author[W. Ejsmont, U. Franz, K. Szpojankowski]{W. Ejsmont, U. Franz, K. Szpojankowski}

\address[W. Ejsmont]{Department of Mathematical Structure Theory (Math C)
TU Graz
Steyrergasse 30, 8010 Graz, Austria and Department of Mathematics and Cybernetics
  Wroclaw University of Economics \\
ul. Komandorska 118/120, 53-345 Wroc\l{}aw, Poland}
\email{wiktor.ejsmont@gmial.com}
\address[U. Franz]{D\'epartement de math\'ematiques de Besan\c{c}on,
Universit\'e de Franche-Comt\'e 16, route de Gray, 25 030
Besan\c{c}on cedex, France}
\email{uwe.franz@univ-fcomte.fr}
\thanks{UF was supported by an ANR Project OSQPI (ANR-11-BS01-0008) and by the Alfried Krupp Wissenschaftskolleg in Greifswald.}
\thanks{WE was  partially supported by  Austrian Science Fund (FWF)
Project No P 25510-N26 and by the  Polish National Science Center grant No. 2012/05/B/ST1/00626}
\thanks{KSz was partially supported by the Polish National Science Center grant No. 2012/05/B/ST1/00554.}
\address[K. Szpojankowski]{Wydzia\l{} Matematyki i Nauk Informacyjnych\\
Politechnika Warszawska\\
ul. Koszykowa 75\\
00-662 Warszawa, Poland}
\email{k.szpojanowski@mini.pw.edu.pl}
\begin{document}
\begin{abstract} 
Characterization problems in free probability are studied here. Using subordination of free additive and free multiplicative convolutions we generalize some known characterizations in free probability to random variables with unbounded support. Using this technique we also prove a new characterization of distributions of free random variables. A similar technique is used to study Laha-Lukacs regressions for monotonically independent random variables.
\end{abstract}
\maketitle

\section{Introduction}

Regressive characterization problems in free probability have been studied in last the years by several authors. The first result in this direction was obtained by Bo\.zejko and Bryc in \cite{BoBr} where Laha-Lukacs regressions in free probability were studied. Ejsmont \cite{Ejs0} proved that measures obtained in \cite{BoBr} have the assumed regressions.
Roughly speaking, Laha-Lukacs regression in classical probability says that only the Meixner distributions have the property that the first conditional moment of $X$ given $X+Y$ is a linear function of $X+Y$, and the second conditional moment of the same type is a quadratic function of $X+Y$. Another characterization of similar type was studied in \cite{SzWes}, where free Poisson and free binomial distributions were characterized by properties of first two conditional moments. For other results in similar directions consult \cite{Ejs,Ejs1,Ejs2,Szp}.

The original motivation for this paper was to generalize and find a unified way to proof known regressive characterizations of distributions of free random variables. To this end we use Biane's results concerning subordination of free additive and free multiplicative convolutions (see \cite{Bian2}). This approach leads us to prove in a simpler way results from \cite{BoBr} and \cite{SzWes}. It allowed us also to relax the condition of compact support for the distribution to the existence of some low order moments. Using the developed technique we also prove a new characterization of free binomial distribution which is a free analogue of the classical characterization of the Beta distribution of the first kind from \cite{SeshWes}.

Additionally we study the Laha-Lukacs regression for monotone independent random variables, and we prove that monotone Meixner distributions can be determined in a similar way as we determine it in free case. It turns out that the free and monotone Meixner classes coincide.

We would like to note that we will work in two settings: for free probability results we assume that the probability space is tracial (i.e.  the state is tracial) and we call such probability space a tracial probability space, in the monotone case we do not assume traciality of the state and we refer to such probability space as non-commutative probability space.

The paper is organized as follows.  In section 2 we introduce briefly the necessary prerequisites on free and monotone probability theory. In section 4 we state the main results and then prove them.

\section{Prerequisites on free and monotone probability} 

\subsection{Free probability}

By a tracial probability space we understand here a pair $(\mathcal{A},\tau)$ where $\mathcal{A}$ is a von Neumann algebra, and  $\tau:\mathcal{A} \to \mathrm{C}$ is a normal, faithful, tracial state, i.e. $\tau(\cdot)$ is linear, continuous in weak* topology, $\tau(\X \Y)=\tau(\Y \X)$,  $\tau(\mathrm{I})=1$, $\tau(\X\X^*)\geq 0$ and $\tau(\X \X^{*}) = 0$ implies $\X = 0$ for all $\X,\Y \in \mathcal{A}$.
We denote by $\widetilde{\mathcal{A}}$ the algebra of unbounded operators affiliated to $\A$ (i.e. $\X \in \widetilde{\mathcal{A}}$ , if and only if all its spectral measures are in $\A$) and
by $\widetilde{\mathcal{A}}_{sa}$ the subspace of self-adjoint elements of $\widetilde{\mathcal{A}}$. 
A (noncommutative)  random variable $\X$ is a self-adjoint (i.e. $ \X=\X^*$) element of $\widetilde{\mathcal{A}}_{sa}$.
If $\X\in\widetilde{\mathcal{A}}_{sa}$, the distribution of $\X$ in the state $\tau$ is the unique probability measure $\mu_\X$ on $\mathbb{R}$ such that $\tau( f(\X))=\int_\mathbb{R}f(\lambda)d\mu_\X(\lambda)$
for any bounded Borel function $f$ on $\mathbb{R}$.
\\ \\
Let $I$ be a set of indices and $\mathcal{B}_i$, for $i \in I$ , be von Neumann subalgebras
of $\mathcal{A}$. Subalgebras $\left(\mathcal{B}_i\right)_{i\in I}$ are called free if $\tau(\X_{1} \dots \X_{n} ) = 0$ whenever $\tau(\X_{j} ) = 0$ for all
$j = 1,\dots, n$ and $\X_{j} \in \mathcal{B}_{i(j)}$ for some indices $i(1)\neq i(2)\neq \dots \neq i(n) .$
 Random variables $\X_{1},\dots ,\X_{n} $  are freely independent (free) if algebras generated by their spectral projections are free. 
\\
\\
If $\mathcal{B} \subset \mathcal{A}$ is a von Neumann subalgebra and $\mathcal{A}$ has a trace $\tau$, then there exists a unique conditional
expectation from $\mathcal{A}$ to $\mathcal{B}$ with respect to $\tau$, which we denote by $\tau(\cdot|\mathcal{B})$. This map
is a weakly continuous, completely positive, identity preserving, contraction and it is characterized by the property that, $\tau(\X\Y) = \tau(\tau(\X|\mathcal{B})\Y )$  for any $\mathbb{X} \in \mathcal{A}$ and for any $\mathbb{Y} \in \mathcal{B}$ (see \cite{Bian2,T}). For fixed $\X \in \mathcal{A}$ by $\tau(\cdot|\X)$ we denote the conditional expectation corresponding to the von Neumann algebra $\mathcal{B}$ generated by $\X$ and $\mathbb{I}$. 
\\ \\
\noindent Let $\mu$ and $\nu$ are be probability measures on $\mathbb{R}$, 
then there exists a tracial
probability space $(\A,\tau)$ and self-adjoint elements $\X,\Y$ in $\widetilde{\mathcal{A}}_{sa}$ with
respective distributions $\mu$ and $\nu$, such that $\X$ and $\Y$ are free. The distribution of $\X+\Y$ is called the free additive convolution of $\mu$ and $\nu$ and is denoted by $\mu \boxplus \nu$.
\\
\\
Assume that $\X$ and $\Y$ are self-adjoint and free, $\X$, $\Y$ with distributions $\mu$ and $\nu$, respectively. If $\X$ is positive, then the support of $\mu$ is a subset of $(0,\infty)$, and the free multiplicative convolution of $\mu$ and $\nu$ is defined as the distribution of $\sqrt{\X}\,\Y\sqrt{\X}$ and denoted by $\mu\boxtimes\nu$.  
For more details about free convolutions and free probability theory, the reader can consult \cite{NicaSpeicherLect,ViculecuDykemaNica}.

\subsection{Analytical Tools}
The Cauchy-Stieltjes transform of a probability measure $\mu$ is defined as
$$
G_{\mu}(z)=\int_{\R}\,\frac{\mu(dx)}{z-x},\qquad \Im(z)>0.
$$
Then define the mapping $L_\mu:\C^+\to\C^+$ by:
$$
L_\mu(z)=\frac{1}{G_{\mu}(z)}
$$
and note that $L_\mu$ is analytic on $\C^+$.
It was proved by Bercovici and Voiculescu in \cite{BercoVoiculecu} that there exist positive numbers $\eta$ and $M$, such that $L_\mu$ has an (analytic) right inverse $L_\mu^{-1}$ defined on the region
$$\Lambda_{\eta,M}=\{z\in\C:|\Re(z)|<\eta\Im(z), \Im(z)>M\}.$$
For a non-commutative random variable $\X$ with distribution $\mu$ its Voiculescu transform $\varphi_\X=\varphi_\mu$ is defined as
\begin{align}\label{rtr}
\varphi_\mu(z)=L^{-1}_\mu(z)-z=\mathcal{R}_\mu(1/z),
\end{align}
on any region of the form $\Lambda_{\eta,M}$, where $L^{-1}$  is defined.
 From \cite{BercoVoiculecu} it follows that for $\X$ and $\Y$ which are freely independent
\begin{align}\label{freeconv}
\varphi_{\X+\Y}(z)=\varphi_{\X}(z)+\varphi_{\Y}(z).
\end{align}
This relation explicitly (in the sense of $\varphi$-transform) defines free convolution of $\X$ and $\Y$.
If $\X$ has the distribution $\mu$, then often we will write $\varphi_{\mu}$ instead $\varphi_{\X}$. Note that by Bercovici and Voiculescu's result in \cite{BercoVoiculecu}, any two of the three distributions of $\X$, $\Y$, and $\X+\Y$ determine uniquely the third.
\\
\\
\noindent
Another analytical tool is an $S$-transform which works nicely with products of freely independent variables. Let  $\nu$ denote the  probability measure supported on $\R^+$. Observe
that, with exception of $\delta_0$, all such measures have nonzero first moment (if the 1st moment exists) and
we assume throughout that we are not dealing with this measure. Consider
the following function:
$$\psi_{\mu}(z)=\int_{\mathbb{R}^+}\frac{z\xi}{1-z\xi}d\mu(\xi),$$
for $z\in\C \backslash\R^+$. As it was proved in \cite{BercoVoiculecu} $\psi_{\mu}|_{i\C^+}$ is univalent and maps into
an open neighborhood about the interval $(\mu(\{0\})-1,0)$. Let $\Omega_\mu=\psi_{\mu}|_{i\C^+}$ and let $\chi: \Omega_\mu=\psi_{\mu}|_{i\C^+}\to i\C^+ $ denote the inverse function.
 For a non-commutative random variable $\X$ its $S$-transform, denoted by $S_{\X}$, is defined by the equation
\begin{align}\label{Str}
S_{\X}(z)=\frac{(1+z)\chi_\mu(z)}{z},
\end{align}
 For $\X$ and $\Y$ which are freely independent
\begin{align}\label{Scon}
S_{\X\,\Y}=S_{\X}\,S_{\Y}. 
\end{align}
\\
\\
\noindent\textbf{Free Meixner distribution}

\noindent A non-commutative random variable $\X$ is said to be a free Meixner variable if the Cauchy-Stieltjes transform is given by the formula

\begin{eqnarray}
G_{\mu_{a,b}}(z)=\frac{(1 + 2b)z + a -\sqrt{(z - a)^2 - 4(1 + b)}}{2(bz^2 + az + 1)} , \label{eq:GtransformataMixner}
\end{eqnarray}
\\
where $a \in \mathbb{R},b\geq -1$ and $|z|$ is big enough, where the branch of the analytic square root should be determined by the condition
that $\Im(z)>0\Rightarrow \Im(G_\mu(z))\leqslant 0$ (see \cite{SY}). 
Equation (\ref{eq:GtransformataMixner}) describes a family of distributions  with mean zero and variance one (see \cite{Ejs1,SY}). 
The absolutely continuous part of $\mu_{a,b}$ is 
$$\frac{\sqrt{4(1+b)-(x-a)^2}}{2\pi(bx^2+ax+1)},$$
on $a -\sqrt{4(1+b)} \leq x \leq a +\sqrt{4(1+b)} $. The measure $\mu_{a,b}$ may also have one atom if $a^2>4b  \geq 0$ and a second atom if $-1 \leq b  < 0$.

Free Meixner distributions have Voiculescu transform
\begin{equation}\label{meixner-phi}
\varphi_{a,b}(z)=\frac{1}{2b}\left(z-a+\sqrt{(z+a)^2 - 4b}\right)
\end{equation}
and R-tranform
\[
\mathcal{R}_{a,b}(z)=\frac{1}{2bz}\left(1-az + \sqrt{(1+az)^2 - 4bz^2}\right).
\]
\\
\\
\\
\\

The free Poisson and free Binomial distributions  belong to the free Meixner class. Since this distributions will appear in several places in this paper, we need to state some additional facts about them.

\textbf{Free Poisson distribution}

\noindent 

A non-commutative random variable $\X$ is said to be free-Poisson variable if it has Marchenko-Pastur (or free-Poisson) distribution $\nu=\nu(\lambda,\alpha)$ defined by the formula
\begin{align} \label{MPdist}
\nu=\max\{0,\,1-\lambda\}\,\delta_0+\lambda \tilde{\nu},
\end{align}
where $\lambda\ge 0$ and the measure $\tilde{\nu}$, supported on the interval $(\alpha(1-\sqrt{\lambda})^2,\,\alpha(1+\sqrt{\lambda})^2)$, $\alpha>0$ has the density (with respect to the Lebesgue measure)
$$
\tilde{\nu}(dx)=\frac{1}{2\pi\alpha x}\,\sqrt{4\lambda\alpha^2-(x-\alpha(1+\lambda))^2}\,dx.
$$
The parameters $\lambda$ and $\alpha$ are called the rate and the jump size, respectively.
It is worth to note that a non-commutative variable with Marchenko-Pastur distribution arises also as a limit in law (in non-commutative sense) of variables with distributions $((1-\frac{\lambda}{N})\delta_0+\frac{\lambda}{N}\delta_{\alpha})^{\boxplus N}$ as $N\to\infty$, see \cite{NicaSpeicherLect}. Therefore, such variables are often called free-Poisson.
If $\X$ is free-Poisson with distribution $\nu(\lambda,\alpha)$ then its $S$-transform has the form (see  \cite{SzWes})
\begin{align}  S_{\X}(z)=\frac{1}{\alpha\lambda+\alpha z}. \label{eq:StransformataPoison}\end{align} 
\\
\\
\\
\textbf{Free binomial distribution}

\noindent A non-commutative random variable $\X$ is free-binomial if its distribution $\beta=\beta(\sigma,\theta)$ is defined by
\begin{align}\label{freebeta}
\beta=(1-\sigma)\mathbb{I}_{0<\sigma<1}\,\delta_0+\tilde{\beta}+(1-\theta)\mathbb{I}_{0<\theta<1}\delta_1,
\end{align}
where $\tilde{\beta}$ is supported on the interval $(x_-,\,x_+)$,
\begin{align}
\label{binom_supp}
x_{\pm}=\left(\sqrt{\frac{\sigma}{\sigma+\theta}\,\left(1-\frac{1}{\sigma+\theta}\right)}\,\pm\,\sqrt{\frac{1}{\sigma+\theta}\left(1-\frac{\sigma}{\sigma+\theta}\right)}\right)^2,
\end{align}
and has the density
$$
\tilde{\beta}(dx)=(\sigma+\theta)\,\frac{\sqrt{(x-x_-)\,(x_+-x)}}{2\pi x(1-x)}\,dx,
$$
where
$(\sigma,\theta)\in \left\{(\sigma,\theta):\,\frac{\sigma+\theta}{\sigma+\theta-1}>0,\,\frac{\sigma\theta}{\sigma+\theta-1}>0\right\}$.
The n-th free convolution power of distribution
$$
p\delta_0+(1-p)\delta_{1/n}
$$
is free-binomial distribution with parameters $\sigma=n(1-p)$ and $\theta=np$, which justifies the name of the distribution (see \cite{SY}). Its $S$-transform is of the form (see  \cite{SzWes})
\begin{align} S_{\U}(z)=1+\frac{1}{\sigma/\theta+ z/\theta}. \label{eq:StransormataBainomial}
\end{align}

\subsection{Complementary facts}

The proofs of the main theorems are based on Biane's \cite{Bian2} description of the conditional expectation of functions of $\X +\Y$ onto the algebra
generated by $\X$.   Here, we introduce the theorems of Biane in a special case i.e. when our function is resolvent type map. The main idea of these results is that for free additive and multiplicative convolutions there exist a subordination functions.
\begin{theo}
Let $(\mathcal{A},\tau)$ be a tracial probability space, $\mathcal{B}$  be a von
Neumann subalgebra of  $\mathcal{A}$, let  $\Y\in \widetilde{\A}_{sa}$, be a self-adjoint element which is free with $\mathcal{B}$, and
let $\X=\X^*\in \widetilde{\mathcal{B}}_{sa}$. Denote by $\mu$ and $\nu$ the distributions of $\X$ and $\Y$ respectively. Then there exists  an analytic function $F$ on $\mathbb{C}\setminus\mathbb{R}$
such that
\begin{align}  
\tau\big(\mathcal{R}_{\X+\Y}(z)|\mathcal{B})=\mathcal{R}_{\X}(F(z)). 
\end{align}
where $\mathcal{R}_{\X}(z)=(z-\X)^{-1}$ (i.e. resolvent map). Function $F$ satisfies $F(\overline{\xi})=\overline{F(\xi)}$, $F(\mathbb{C}^+)\subset\mathbb{C}^+$, $Im(F(\xi))\geq Im(\xi)$, $\lim_{y\to\infty} \frac{F(iy)}{iy}=1$ and we have
$$G_{\mu\boxplus\nu}(\xi)=G_\mu(F(\xi)),$$
for all $\xi \in \mathbb{C}^+$.
\label{twr:1}
\end{theo}

\begin{Rem}\label{rem-sub}
If $\varphi_\mu$ and $\varphi_\nu$ are proportional, i.e., if $\nu$ is a free
convolution power of $\mu$, then we can give an explicit expression for the
subordination function. Suppose we have $\varphi_\nu(z) = \lambda
\varphi_\mu(z)$ with some $\lambda>0$, i.e., $\nu=\mu^{\boxplus \lambda}$,
then we have $\varphi_{\mu\boxplus\nu}=(1+\lambda)\varphi_\mu$ and therefore
\[
L^{-1}_\mu(z) = \frac{1}{1+\lambda} L^{-1}_{\mu\boxplus\nu}(z) + \frac{\lambda
  z}{1+\lambda}.
\]
The subordination function is then given by
\[
F(z) = L^{-1}_\mu\circ L_{\mu\boxplus\nu}(z) = \frac{\lambda}{1+\lambda}
L_{\mu+\nu}(z) + \frac{z}{1+\lambda}
\]
i.e., it is the reciprocal Cauchy-Stieltjes transform of a boolean convolution
power or $\mu\boxplus\nu$. More precisely, we have $F=L_\rho$, with
$\rho=(\mu\boxplus\nu)^{\uplus \frac{\lambda}{1+\lambda}}$, where $\uplus$
denotes the additive boolean convolution, cf.\ \cite{SpWo}.
\end{Rem}

\noindent We also need the following theorem in the case of multiplicative free convolution.  

\begin{theo}
Let $(\mathcal{A},\tau)$ be a tracial probability space, $\mathcal{B}$  be a von
Neumann subalgebra of  $\mathcal{A}$, and  $\X,\Y\in \widetilde{\A}_{sa}$ such that $\X$ and $\Y$ are positive, with
respective distributions $\mu$  and $\nu$, different from $\delta_0$, one has $\X\in \widetilde{\mathcal{B}}_{sa}$ and $\Y$ is free with $\B$. Then there exists  an analytic function $F$ on $\mathbb{C}\setminus\mathbb{R}_+$ 
such that

\begin{align}  
\tau\big(\xi\X^{1/2}\Y\X^{1/2}(\I-\xi\X^{1/2}\Y\X^{1/2})^{-1}|\mathcal{B})=F(\xi)\X(\I-F(\xi)\X)^{-1}, \label{eq:rownanieDualne}
\end{align}
Function $F$ satisfies $F(\overline{\xi})=\overline{F(\xi)}$, $F(\xi)\in\mathbb{C}^+$, $Arg(F(\xi))\geq Arg(\xi)$ and 
$$\psi_{\mu\boxtimes\nu}(\xi)=\psi_\mu(F(\xi)),$$
for all $\xi \in \mathbb{C}^+$.  
\label{twr:2}
\end{theo}
\noindent  We will also need the following fact.
\begin{lemm}
 For $z\in \C \backslash\R^+$ we define the operator $\Psi_{\X}(z)=z\X(\I-z\X)^{-1}$. This operator satisfies the following properties
\begin{align}
\label{eq:psi_prod}
\X^{n}\Psi_\X(z)=\frac{1}{z^n}\left(\Psi_\X(z)-\sum_{i=1}^n z^i\X^i\right),
\end{align}
for all integers $n\geq 0$. \label{lem:1}
\end{lemm}
\begin{proof}
We will prove it inductively, let us check it for $k=1$.
\begin{align*}
\frac{1}{z}\left(z\X\left(\I-z\X\right)^{-1}-z\X\right)=\X\left(\left(\I-z\X\right)^{-1}-\I\right)=
\X\left(\I-z\X\right)^{-1}\left(\I-\left(\I-z\X\right)\right)=z\X^{2}\left(\I-z\X\right)^{-1}=\X\Psi_\X(z).
\end{align*}
Assume now that \eqref{eq:psi_prod} holds for $n=k-1$ we will prove it for $n=k$.
\begin{align*}
\X^{k}\Psi_\X(z)&=\X\X^{k-1}\Psi_\X(z)=\X\frac{1}{z^{k-1}}\left(\Psi_\X(z)-\sum_{i=1}^{k-1} z^i\X^i\right)=\frac{1}{z^{k-1}}\left(\X\Psi_\X(z)-\sum_{i=1}^{k-1} z^i\X^{i+1}\right)\\
&=\frac{1}{z^{k-1}}\left(\frac{1}{z}\left(\Psi_\X(z)-z\X\right)-\sum_{i=1}^{k-1} z^i\X^{i+1}\right)=\frac{1}{z^k}\left(\Psi_\X(z)-\sum_{i=1}^k z^i\X^i\right),
\end{align*}
which proves \eqref{eq:psi_prod}.\end{proof}

\subsection{Monotone independence}

For working with monotone independence we have to modify our setting. By a
non-commutative probability space 
 we understand here a pair $(\mathcal{A},\phi)$ where $\mathcal{A}$ is a von Neumann algebra, and  $\phi:\mathcal{A} \to \mathbf{C}$ is a normal faithful state,

Let $I$ be a \emph{linearly ordered} set of indices and $\mathcal{B}_i$, for
$i \in I$, be strongly closed *-subalgebras of $\mathcal{A}$.  Subalgebras
$\left(\mathcal{B}_i\right)_{i\in I}$ are called \emph{monotonically
  independent} if
\begin{itemize}
\item[(i)]
$ \X_1\Y\X_2 = \phi(\Y)\X_1 \X_2$ whenever $\X_1,\in B_{i_1}$ $\X_2\in B_{i_2}$,
$\Y\in B_j$ with $j>i_1$ and $j>i_2$;
\item[(ii)]
$\phi(\X_1\cdots \X_k \Y \Z_1 \cdots \Z_{k'}) = \phi(\X_1)\cdots \phi(\X_m)
\phi(\Y) \phi(\Z_1) \cdots \phi(\Z_{m'})$ whenever $\X_\ell\in B_{i_\ell}$, $\Y\in
B_j$, $\Z_\ell\in B_{k_\ell}$ with $i_1 > \cdots > i_m > j$ and $j < k_1 < \cdots < k_{m'}$,
\end{itemize}
cf.\ \cite{mu1,mu2}.

\begin{Rem}
To get interesting examples we have to allow non-tracial states and non-unital
subalgebras. If $\mathcal{B}_1,\mathcal{B}_2\subseteq \mathcal{A}$ are two
unital subalgebras that are monotonically independent w.r.t.\ a state $\phi$,
then we get
\[
\phi(\X_1\X_2) = \phi(\X_1\mathbf{1} \X_2)=\phi(\X_1)\phi(\mathbf{1})\phi(\X_2) 
\]
for all $\X_1,\X_2\in \mathcal{B}_1$, by viewing $\mathbf{1}$ as an element of
$\mathcal{B}_2$. Similarly, if $\phi$ is a trace, then we have
\[
\phi(\X_1)\phi(\Y )\phi(\X_2) = \phi(\X_1\Y \X_2)=
\phi(\X_2\X_1\Y)=\phi(\X_2\X_1)\phi(\Y)
\]
for all $\X_1,\X_2\in \mathcal{B}_1$, $\Y\in \mathcal{B}_2$.

In both cases it follows that the states on all but the last algebra have to
be characters.
\end{Rem}

The non-unitality of the subalgebras also means that we can not have
(unit-preserving) conditional expectations. We call a map
$\E:\mathcal{A}\to\mathcal{B}$ from a non-commutative probability space
$(\mathcal{A},\phi)$ onto a strongly closed *-subalgebra a \emph{generalized
  conditional expectation}, if it satisfies the following conditions:
\begin{itemize}
\item[(i)]
$\E$ preserves $\phi$, i.e.\ $\phi\circ \E=\phi$;
\item[(ii)]
$\E$ is idempotent, i.e.\ $\E\circ\E=\E$;
\item[(iii)]
$\E$ is completely positive;
\item[(iv)]
$\E$ is a contraction, i.e.\ $||\E||\le 1$;
\item[(v)]
$\E$ satisfies the module property $\E(\X\Y\Z) = \X\E(\Y)\Z$, for
$\X,\Y\in\mathcal{B}$, $\Z\in \mathcal{A}$.
\end{itemize}

A typical example is given by the following. Let $H$ be a Hilbert space with
unit vector $\omega\in H$ and $P\in \mathcal{B}(H)$ an orthogonal projection
that leaves $\omega$ invariant. Set $\mathcal{A}=\mathcal{B}(H)$,
$\phi(\X)=\langle \omega, \X\omega\rangle$ and $\mathcal{B}=\{P\X P; \X\in
\mathcal{B}\}$. Then $\E(\X)=PXP$ for $\X \in \mathcal{A}$ defines a generalized
conditional expectation onto $\mathcal{B}$.

\begin{theo}\label{thm-con-exp-mon}
Let $(\mathcal{A},\phi)$ be a non-commutative probability space,
$\mathcal{B}_1,\mathcal{B}_2$ two monotonically independent strongly closed
*-subalgebras of  $\mathcal{A}$. Then there exists a generalized conditional expectation
\[
\E_\phi(\cdot|\mathcal{B}_1) : \mathcal{B}_1\vee\mathcal{B}_2\to  \mathcal{B}_1
\]
where $\mathcal{B}_1\vee\mathcal{B}_2$ denotes the smallest strongly closed
*-subalgebra of  $\mathcal{A}$ containing $\mathcal{B}_1$ and $\mathcal{B}_2$.
\end{theo}
\begin{proof}
Denote by $(H,\pi,\xi)$ the GNS representation of
$(\mathcal{B}_1\vee\mathcal{B}_2,\phi|_{\mathcal{B}_1\vee\mathcal{B}_2})$, and
set $H_1=\overline{\pi(\mathcal{B}_1)\xi}$,
  $H_2=\overline{\pi(\mathcal{B}_2)\xi}$. Since $\mathcal{B}_1$ and
    $\mathcal{B}_2$ are monotonically independent, one can show that
\[
U:\pi(\X_2)\pi(\X_1)\xi \mapsto \pi(\X_1)\xi \otimes \pi(\X_2)\xi
\]
extends to a unitary map from $H$ to $H_1\otimes H_2$, by a calculation similar to the proof of \cite[Theorem 3.5]{fr}. Furthermore, $U$ implements a unitary equivalence between the restriction of $\pi$ and the representation $\rho :
\mathcal{B}_1\vee\mathcal{B}_2\to \mathcal{B}(H_1\otimes H_2)$ determined by
\begin{eqnarray*}
\rho(\X) &=& \pi(\X)\otimes  P_\xi \qquad \mbox{ for } \X \in
\mathcal{B}_1, \\
\rho(\X) &=& {\rm id}_{H_1} \otimes \pi(\X) \qquad \mbox{ for } \X \in \mathcal{B}_1,
\end{eqnarray*}
where $P_\xi$ denotes the orthogonal projection onto $\xi$. I.e., the algebra
$\rho(\mathcal{B}_1)$ acts as zero on the orthogonal complement of $H_1\otimes
\xi\cong H_1$ and is a von Neumann subalgebra of $\mathcal{B}(H_1)$, when
restricted to this subspace. It follows that $\rho(\mathcal{B}_1)$ contains
the operator $\tilde{P}={\rm id}_{H_1}\otimes P_\xi$. Set
$P=\rho^{-1}(\tilde{P})$, then one can check that $E(\X)=P\X P$ has all the
desired properties.
\end{proof}

The following theorem shows how to compute this conditional expectation in the
situations which we will need here, it can be proved as in \cite[Proposition 3.2]{fm}.

\begin{theo}\label{thm-sub-mon}
Let $(\mathcal{A},\phi)$ be a non-commutative probability space, $\mathcal{B}$
be a strongly closed *-subalgebra of  $\mathcal{A}$ , let  $\Y\in
\widetilde{\A}_{sa}$, be an self-adjoint element which is monotonically
independent with $\mathcal{B}$, and let $\X=\X^*\in
\widetilde{\mathcal{B}}_{sa}$. Denote by $F=L_\Y$ the reciprocal Cauchy-Stieltjes transform of $\Y$.
\begin{align}  
\E_\phi\big(\mathcal{R}_{\X+\Y}(z)|\mathcal{B})=\mathcal{R}_{\X}(F(z)). 
\end{align}
where $\mathcal{R}_{\X}(z)=(z-\X)^{-1}$ (i.e. resolvent map).
\end{theo}
We will use the notation $\E_\phi(\cdot|\X)$ for the generalized
conditional expectation onto the strongly closed subalgebra generated by a
r.v.\ $\X$, but note that here this algebra will not be unital in general.

Muraki \cite{mu1} showed the additive monotone convolution is described by the composition of the reciprocal Cauchy-Stieltjes transforms, If $\X$ and $\Y$ are two monotonically independent self-adjoint operators with reciprocal Cauchy-Stieltjes transforms $L_\X$ and $L_\Y$, then $\X+\Y$ has the reciprocal Cauchy-Stieltjes transform $L_{\X+\Y}=L_\X\circ L_\Y$, and the Cauchy-Stieltjes transform
\begin{align}\label{mon-conv}
G_{\X+\Y}= G_\X\circ L_\Y.
\end{align}
This result was extended to posssibly unbounded, essentially self-adjoint operators in \cite{fr}. Note that Bercovici and Voiculescu's result in \cite{BercoVoiculecu} implies also in the case of monotone independence that any two of the three distributions of $\X$, $\Y$, and $\X+\Y$ uniquely the third.

\section{The main results and proofs of the theorems}

The following theorems are the main results of the paper. The first theorem was proved by Bo\.zejko and Bryc in \cite{BoBr}, authors used combinatorial approach and had to assume existence of all moments. Our proof is based on Biane's subordination trick of the conditional expectation of functions of $\X +\Y$ given by the algebra generated by $\X$. This allows us to relax conditions on moments, we assume only the existence of variance of random variables, which is natural assumption. The idea of this proof is close to the original Laha and Lukacs proof of the theorem \cite{LL}, with characteristic functions replaced by Cauchy-Stieltjes transforms. 

\begin{theo}[]
Suppose that $\X$,$\Y \in \widetilde{\A}_{sa}$  are free, self-adjoint, non-degenerate, centered $(\tau(\X)=\tau(\Y)=0)$, $\tau(\X^2+\Y^2)=1$ and there are numbers $\alpha,\beta>0$ and $\alpha+\beta =1$ such that  
\begin{align} 
\tau(\X|\X+\Y) =\alpha(\X+\Y) \label{eq:warunkowypierwszymment}\end{align}  
and
\begin{align} 
Var(\X|\X+\Y) =
\frac{\alpha \beta}{(b+1)}
 \big[\mathrm{I}+a(\X+\Y)+b(\X+\Y)^{2}\big] .  \label{eq:warjancjawarunkowa2}
\end{align}  
Then $\X/\sqrt{\alpha}$ and  $\Y/\sqrt{\beta}$ have the free Meixner  laws $\mu_{a/\sqrt{\alpha},b/\alpha}$ and  $\mu_{a/\sqrt{\beta},b/\beta}$, respectively. 
\label{twr:3}
 \end{theo}

The next theorem gives an analog of this result for monotone independence. For
monotonically independent subalgebras we have only generalized conditional
expectations onto the past, therefore we have to reformulate the regression
property. We say that $\X$ and $\Y$ satisfy a quadratic regression with
parameters $\alpha$, $\beta$, $a$, and $b$, if
\begin{align} 
\phi\big(\X f(\X+\Y)\big) =\alpha\phi\big((\X+\Y)f(\X+\Y)\big)
\end{align}  
and
\begin{align}  
\phi\big(\X f(\X+\Y)\X\big) = \phi\Big((\alpha \beta\big(
 \mathrm{I}+a(\X+\Y)+b(\X+\Y)^{2}\big)/(b+1)+\alpha^2(\X+\Y)^2)f(\X+\Y)\Big)
\end{align}
for all bounded continuous functions $f\in C_b(\mathbb{R})$. Note that it follows from the definition of monotone independence that we have
\[
\phi\big(\X^2 f(\X+\Y)\big) = \phi\big(\X f(\X+\Y)\X\big) = \phi\big( f(\X+\Y)\X^\big).
\]

When the conditional expectation of $\X$ and $\X^2$ onto the algebra generated
by $\X+\Y$ exists, then the quadratic regression property defined above is
equivalent to the conditions \eqref{eq:warunkowypierwszymment} and \eqref{eq:warjancjawarunkowa2} we used in the previous theorem.

\begin{theo}[]
Suppose that $\X$,$\Y \in \widetilde{\A}_{sa}$  are monotonically independent,
self-adjoint, non-degenerate, centered $(\tau(\X)=\tau(\Y)=0)$,
$\tau(\X^2+\Y^2)=1$ and that they satisfy a quadratic regression with
parameters $\alpha,\beta>0$, $\alpha+\beta =1$.

Then $\X/\sqrt{\alpha}$ and  $(\X+\Y)$ have the free Meixner  laws $\mu_{a/\sqrt{\alpha},b/\alpha}$ and  $\mu_{a,b}$, respectively. 
\label{mono}
 \end{theo}

\begin{Rem}
Because of the ``non-unitality'' of monotone independence, the assumption that
$\X$ is centered is more important than in the free case, where it could
easily be removed.
\end{Rem}

Next theorem gives a similar generalization of a result from \cite{SzWes} where authors proved characterization of free binomial and free Piosson random variable. This result is a free analogue of classical probability characterization proved in \cite{BobWes}.
\begin{theo}
\label{twr:4}
Suppose that $\U,\V \in \widetilde{\A}_{sa}$ are free, self-adjoint, non-degenerate and $\V,\U$ have a distribution  supported on $(0,\infty]$ and  distribution of $\U$ is supported on $[0,\infty]$. Assume that there exist real constants $c$ and $d$ such that
\begin{align}
\tau \left(\left.\V-\V^{\frac{1}{2}}\,\U\,\V^{\frac{1}{2}}\right|\V^{\frac{1}{2}}\,\U\,\V^{\frac{1}{2}}\right)=c\,\I \label{eq:warunkowypierwszymmentDualny}
\end{align}
and
\begin{align}
\tau\left(\left.(\V-\V^{\frac{1}{2}}\,\U\,\V^{\frac{1}{2}})^2\right|\V^{\frac{1}{2}}\,\U\,\V^{\frac{1}{2}}\right)=d\,\I.\label{eq:warjancjawarunkowaDualna}
\end{align}
Then  $\V$ has free-Poisson distribution, $\nu(\lambda,\alpha)$ with $\lambda=\sigma+\theta$, $\left(\sigma=\frac{c\big(\tau(\V\U)+2c-2\tau(\V)\big)}{c^2-d},\theta=\frac{c^2}{d-c^2 }\right)$, $\alpha=\frac{d-c^2}{c}$ and $\U$ has free-binomial distribution,  $\beta(\sigma,\theta)$.
\end{theo}
The next result is a free analogue of characterization of the Beta distribution of the first kind proved by Weso\l{}owski and Seshadri in \cite{SeshWes}.
\begin{theo}
\label{tw:binomal}
Let $\X,\Y\in \widetilde{\A}_{sa}$ be free, selfadjoint and non-degenerate. Moreover assume that 
\begin{align}
\label{eq:binomal1}
\tau\left(\I-\Y|\Y^{\frac{1}{2}}\X\Y^{\frac{1}{2}}\right)=&c\left(\I-\Y^{\frac{1}{2}}\X\Y^{\frac{1}{2}}\right)
\\
\label{eq:binomal2}
\tau\left(\left(\I-\Y\right)^{-1}|\Y^{\frac{1}{2}}\X\Y^{\frac{1}{2}}\right)=&d\left(\I-\Y^{\frac{1}{2}}\X\Y^{\frac{1}{2}}\right)^{-1}
\end{align} 
for some real constants $c,d$.

Then $\X$ has free binomial distribution with parameters  $\sigma_\X=\frac{(1-c)d\alpha}{cd-1}>0$ and $\theta_\X=\frac{(c-1)(d-1)}{1-cd}>0$, and $\Y$ has free binomial distribution with parameters  $\sigma_\Y=\frac{(1-c) (d(\alpha+1)-1)}{cd-1}>0$ and $\theta_\Y=\frac{c(1-d)}{1-cd}>0$.
\end{theo}

We now present the proof of our results.

\subsection{Proof of Theorem \ref{twr:3}}
\begin{proof}
It is easy to check that 
\begin{align}  
&\X\mathcal{R}_{\X}(z)=-\I+z\mathcal{R}_{\X}(z) , \label{eq:resowenta1}\\&
\X^2\mathcal{R}_{\X}(z)=-z-\X+z^2\mathcal{R}_{\X}(z). \label{eq:resowenta2}
\end{align}  
We will first compute $\tau(\X\mathcal{R}_{\X+\Y}(z))$ in two different ways.
\\ Let us first observe that 
\begin{align}  
&\tau(\X\mathcal{R}_{\X+\Y}(z))=\tau\big(\tau(\X\mathcal{R}_{\X+\Y}(z)|\X+\Y)\big)=
\tau\big(\tau(\X|\X+\Y)\mathcal{R}_{\X+\Y}(z)\big)\nonumber\\&\stackrel{(\ref{eq:warunkowypierwszymment}) }{=}\alpha\tau\big((\X+\Y)\mathcal{R}_{\X+\Y}(z)\big)\stackrel{(\ref{eq:resowenta1})}{=}
\alpha\big(zG_{\X+\Y}(z)-1\big).
\end{align}
Using Theorem \ref{twr:1} we compute $\tau(\X\mathcal{R}_{\X+\Y}(z))$ as follows  
\begin{align}  
&\tau(\X\mathcal{R}_{\X+\Y}(z))=\tau\big(\tau(\X\mathcal{R}_{\X+\Y}(z)|\X)\big)\nonumber\\&=\tau\big(\X\tau(\mathcal{R}_{\X+\Y}(z)|\X)\big)=\tau\big(\X\mathcal{R}_{\X}(F(z))\big)=F(z)G_\X\big(F(z)\big)-1.
\end{align}
Similar argument applied to the expression $\tau(\X^2\mathcal{R}_{\X+\Y}(z))$ give us 
\begin{align}  
&\tau(\X^2\mathcal{R}_{\X+\Y}(z))=\tau\big(\tau(\X^2\mathcal{R}_{\X+\Y}(z)|\X+\Y)\big)=\tau\big(\tau(\X^2|\X+\Y)\mathcal{R}_{\X+\Y}(z)\big)\\&=\tau\big(\big[\alpha \beta(
 \mathrm{I}+a(\X+\Y)+b(\X+\Y)^{2})/(b+1)+\alpha^2(\X+\Y)^2\big]\mathcal{R}_{\X+\Y}(z)\big)\\&\nonumber\stackrel{(\ref{eq:warjancjawarunkowa2})}{=}
\Big(\alpha\beta\big(G_{\X+\Y}(z)+azG_{\X+\Y}(z)-a-zb+z^2bG_{\X+\Y}(z)\big)/(b+1)+\alpha^2(-z+z^2G_{\X+\Y}(z))\Big).
\end{align}
and
\begin{align}  
&\tau(\X^2\mathcal{R}_{\X+\Y}(z))=\tau\big(\X^2\tau(\mathcal{R}_{\X+\Y}(z)|\X)\big)=\tau\big(\X^2\mathcal{R}_{\X}F(z)\big)\stackrel{(\ref{eq:resowenta2})}{=}-F(z)+F^2(z)G_\X\big(F(z)\big).
\end{align}
This gives the system of equations
\begin{align}&  
\alpha\big(zG_{\X+\Y}(z)-1\big)=F(z)G_\X\big(F(z)\big)-1, \label{eq:pom1dowod1}
\\ 
&\frac{\alpha\beta\big((1+az+z^2b)G_{\X+\Y}(z)-a-zb\big)}{b+1}+\alpha^2(-z+z^2G_{\X+\Y}(z)) =-F(z)+F^2(z)G_\X\big(F(z)\big) \label{eq:pom2dowod1}
 \end{align}
Finally, \eqref{eq:pom1dowod1} and \eqref{eq:pom2dowod1} together with $G_{\X+\Y}(z)=G_{\X}(F(z))$ give 

$$(1+za+bz^2)G^2_{\X+\Y}(z) -(z+a+2bz)G_{\X+\Y}(z)+1+b=0,$$
which has the solution
\begin{eqnarray}
G_{\X+\Y}(z)=\frac{(1 + 2b)z + a -\sqrt{(z - a)^2 - 4(1 + b)}}{2(bz^2 + az + 1)}. \label{eq:GtransformataMixnerDowod}
\end{eqnarray}
because $\Im(z)>0\Rightarrow \Im(G_\mu(z))\leqslant 0$, which means that $\X+\Y$ has free Meixner law. 
From this we also deduce that $\X$ and $\Y$ are bounded random variable and we can proceed analogously to the proof of Theorem 3.2 or  2.3   from \cite{BoBr} and \cite{Ejs}, respectively. 

\end{proof}

\subsection{Proof of Theorem \ref{mono}}

\begin{proof}
We compute again the expections of $\X\mathcal{R}_{\X+\Y}(z)$ and
$\X^2\mathcal{R}_{\X+\Y}(z)$ in two different ways. Applying the quadratic
regression property and Theorem \ref{thm-sub-mon} to $\X\mathcal{R}_{\X+\Y}(z)$, we get
\[
\phi(\X\mathcal{R}_{\X+\Y}(z)\big) =
\phi\left(\alpha\phi\big((\X+\Y)\mathcal{R}_{\X+\Y}(z)\big)\right)
= \alpha\big(G_{\X+\Y}(z)-1\big)
\]
and 
\[
\phi(\X\mathcal{R}_{\X+\Y}(z)\big) =
\phi\Big(\X\E_\phi\big(\mathcal{R}_{\X+\Y}(z)|\X\big)\Big) =
\phi\Big(\X\mathcal{R}_X\big(L_\Y(z)\big)\Big) = L_\Y(z) G_\Z\big(L_\Y(z)\big)-1.
\]
For $\X^2\mathcal{R}_{\X+\Y}(z)$ we obtain similarly,
\begin{align}
& \phi\big(\X^2\mathcal{R}_{\X+\Y}(z)\big) = \phi\big(\big[\alpha \beta(
 \mathrm{I}+a(\X+\Y)+b(\X+\Y)^{2})/(b+1)+\alpha^2(\X+\Y)^2\big]\mathcal{R}_{\X+\Y}(z)\big)  \\
& = \Big(\alpha\beta\big(G_{\X+\Y}(z)+azG_{\X+\Y}(z)-a-zb+z^2bG_{\X+\Y}(z)\big)/(b+1)+\alpha^2\big(-z+z^2G_{\X+\Y}(z)\big)\Big)
\end{align}
and
\[
\phi(\X^2\mathcal{R}_{\X+\Y}(z))=\phi\big(\X^2\mathcal{R}_{\X+\Y}(z)|\X)\big)=\phi\big(\X^2\mathcal{R}_{\X}\Big(L_Y(z)\big)\Big)
= -L_\Y(z)+L_\Y^2(z)G_\X\big(L_\Y(z)\big).
\]
We conclude that $G_\X$, $G_{\X+\Y}$ and $F=L_Y$ satisfy again the system of
equations
\eqref{eq:pom1dowod1} and \eqref{eq:pom2dowod1}. Since we also have again
$G_{\X+\Y}=G_\X\circ F$, see Equation \eqref{mon-conv}, we can solve the system in the same way as in the proof of the previous theorem. In the first step we see that $G_{\X+\Y}$ is again given by Equation \eqref{eq:GtransformataMixnerDowod}, i.e.\ $\X+\Y$ has free Meixner distribution $\mu_{a,b}$.

Using $G_\X\circ F= G_{\X+\Y}$ and substituting the formula for $G_{\X+\Y}$ into Equation \eqref{eq:pom1dowod1}, we get $F$. The conditions on $F$, which was in the free case the subordination function in Biane's theorem, imply that it is the reciprocal Cauchy-Stieltjes transform of some probability measure on $\mathbb{R}$. By Bercovici and Voiculescu's result \cite{BercoVoiculecu} $F$ is invertible on some appropriate domain, therefore $G_{\X+\Y}$ and $F$ determine $G_\X$. We see that $\X$ has free Meixner law $\mu_{a/\sqrt{\alpha},b/\alpha}$.
\end{proof}

\begin{Rem}
The only difference between the free and the monotone characterisations of the free Meixner law lies in the distribution of $\Y$. In the free case $F$ is the subordination function and $G_\Y$ is computed from $G_\X$ and $G_{\X+\Y}$. In the monotone case $F$ is the reciprocal Cauchy-Stieltjes transform $L_\Y$ of $\Y$. From Equation \eqref{eq:pom1dowod1} we get
\begin{eqnarray*}
L_\Y(z) &=& \alpha z + (1-\alpha)L_{\X+\Y}(z)\\
&=& \frac{1}{2(1+b)}\left((1+\alpha+2b)z
  + \beta a + \beta\sqrt{(z-a)^2-4(1+b)}\right)
\end{eqnarray*}
and therefore
\[
G_\Y(z)= \frac{(1+\beta+2b)z + \beta a - \beta\sqrt{(z-a)^2-4(1+b)}}{4\big((\alpha+b)z^2 + a \beta z +1]\big)}.
\]
The law of $\Y$ the $\beta$-th boolean convolution power of the law of $\X+\Y$, cf.\ \cite{SpWo}, which is also clear from Remark \ref{rem-sub}. Anshelevich \cite{An2} showed that the Boolean Meixner family also coincides with the free Meixner family, but it fails to satisfy a Boolean Laha-Lukacs-type characterization.
\end{Rem}

\subsection{Proof of Theorem \ref{twr:4}}
\begin{proof} The main idea of the proof is similar to the proof of the Theorem \ref{twr:3}.\\
From Lemma \ref{lem:1} we have  
\begin{align} &\X\Psi_{\X}(z)=-\X+\Psi_{\X}(z)/z, \label{eq:pom1dowod2}
\\  &\X^2\Psi_{\X}(z)=-\X^2-\X/z+\Psi_{\X}(z)/z^2.\label{eq:pom2dowod2}
\end{align}  
 As in the proof of Theorem \ref{twr:3}, we will calculate conditional expectation in two different ways. \\
Let us first proceed with
$\tau(\V\Psi_{\V^{\frac{1}{2}}\,\U\,\V^{\frac{1}{2}}}(z))$. 
We have 
\begin{align}  
&\tau(\V\Psi_{\V^{\frac{1}{2}}\,\U\,\V^{\frac{1}{2}}}(z))
=\tau\big(\tau(\V\Psi_{\V^{\frac{1}{2}}\,\U\,\V^{\frac{1}{2}}}(z)|\V^{\frac{1}{2}}\,\U\,\V^{\frac{1}{2}})\big)=
\tau\big(\tau(\V|\V^{\frac{1}{2}}\,\U\,\V^{\frac{1}{2}})\Psi_{\V^{\frac{1}{2}}\,\U\,\V^{\frac{1}{2}}}(z)\big)
\nonumber\\&\stackrel{(\ref{eq:warunkowypierwszymmentDualny})}{=} \tau\big([c+\V^{\frac{1}{2}}\,\U\,\V^{\frac{1}{2}}]\Psi_{\V^{\frac{1}{2}}\,\U\,\V^{\frac{1}{2}}}(z)\big)\stackrel{(\ref{eq:pom1dowod2})}{=}
c\psi_{\V^{\frac{1}{2}}\,\U\,\V^{\frac{1}{2}}}(z)-\tau(\V\U)+\psi_{\V^{\frac{1}{2}}\,\U\,\V^{\frac{1}{2}}}(z)/z.
\end{align}
Using Theorem \ref{twr:4}, it follows that
\begin{align}  
&\tau(\V\Psi_{\V^{\frac{1}{2}}\,\U\,\V^{\frac{1}{2}}}(z))=\tau\big(\tau(\V\Psi_{\V^{\frac{1}{2}}\,\U\,\V^{\frac{1}{2}}}(z)|\V)\big)\nonumber
\\&=\tau\big(\V\tau(\Psi_{\V^{\frac{1}{2}}\,\U\,\V^{\frac{1}{2}}}(z)|\V)\big)=\tau\big(\V\Psi_{\V}(F(z))\big)=\psi_\V\big(F(z)\big)/F(z)-\tau(\V).
\end{align}
By a similar argument we get
\begin{align}  
&\nonumber\tau(\V^2\Psi_{\V^{\frac{1}{2}}\,\U\,\V^{\frac{1}{2}}}(z))
=\tau\big(\tau(\V^2|\V^{\frac{1}{2}}\,\U\,\V^{\frac{1}{2}})\Psi_{\V^{\frac{1}{2}}\,\U\,\V^{\frac{1}{2}}}(z)\big)
\stackrel{(\ref{eq:warjancjawarunkowaDualna})}{=}\\\nonumber&\tau\big(\big[
 d\mathrm{I}+2(c+\V^{\frac{1}{2}}\,\U\,\V^{\frac{1}{2}})\V^{\frac{1}{2}}\,\U\,\V^{\frac{1}{2}}-(\V^{\frac{1}{2}}\,\U\,\V^{\frac{1}{2}})^2]\Psi_{\V^{\frac{1}{2}}\,\U\,\V^{\frac{1}{2}}}(z)\big)\\&\nonumber\stackrel{(\ref{eq:pom2dowod2})}{=}
d\psi_{\V^{\frac{1}{2}}\,\U\,\V^{\frac{1}{2}}}(z)+2c\psi_{\V^{\frac{1}{2}}\,\U\,\V^{\frac{1}{2}}}(z)/z-2c\tau(\V\U)-\tau((\V\U)^2)-\tau(\V\U)/z+\psi_{\V^{\frac{1}{2}}\,\U\,\V^{\frac{1}{2}}}(z)/z^2
\end{align}
and 
\begin{align}  
&\nonumber \tau(\V^2\Psi_{\V^{\frac{1}{2}}\,\U\,\V^{\frac{1}{2}}}(z))=\tau\big(\V^2\tau(\Psi_{\V^{\frac{1}{2}}\,\U\,\V^{\frac{1}{2}}}(z)|\V)\big)=\tau\big(\V^2\Psi_{\V}F(z)\big)=-\tau(\V^2)-\frac{\tau(\V)}{F(z)}+\psi_\V\big(F(z)\big)/F^2(z).
\end{align}
Taking into account that $\psi_{\V^{\frac{1}{2}}\,\U\,\V^{\frac{1}{2}}}(z)=\psi_{\V\boxtimes\U}(z)=\psi_{\V}(F(z))$ and substituting  $z=F^{-1}(z)$  we obtain two equations  (after simple calculation)  
\begin{align} \nonumber
\psi_{\V^{\frac{1}{2}}\,\U\,\V^{\frac{1}{2}}}(z)^2 z \left(d-c^2\right)-\psi_{\V^{\frac{1}{2}}\,\U\,\V^{\frac{1}{2}}}(z) ( c \tau(V) z
   +\tau(\V)-\tau(\V^2) z-\tau(\V\U)+\tau((\V\U)^2) z)\\\nonumber +\tau(\V\U) z ( \tau(V)-\tau(\V\U))=0,\end{align} 
and
\begin{align}  \nonumber \psi_{\V}(z)^2 \left(d z-c^2 z\right)+\psi_{\V}(z) (-c z \tau(\V\U)+z \tau(\V^2)-z \tau((\V\U)^2) +3
   \tau(V)+\tau(\V\U))\\ \nonumber+z\left( \tau(V)^2+ \tau(\V) \tau(\V\U)\right)=0.\end{align} 
Under the assumptions of Theorem \ref{twr:4}, we have $\tau(\V)-\tau(\U\V)=c$ and $\tau(\V^2)-\tau((\V\U)^2)=d+2c\tau(\U\V)$. Using the relationship between the parameters i.e. $\lambda=\sigma+\theta$, $\sigma=\frac{c\big(\tau(\V\U)+2c-2\tau(\V)\big)}{c^2-d},\theta=\frac{c^2}{d-c^2 }$, $\alpha=\frac{d-c^2}{c}$ we can simplify the above equations
$$\alpha z\psi^2_{\V^{\frac{1}{2}}\,\U\,\V^{\frac{1}{2}}}(z)
 -\psi_{\V^{\frac{1}{2}}\,\U\,\V^{\frac{1}{2}}}(z) (1+ (c -\alpha (1+\lambda))z)-(c-\alpha\lambda)z=0,$$
and
$$\psi_{\V}(z)^2\alpha z +\psi_{\V}(z)(\alpha z +\alpha\lambda z-1)+\alpha\lambda z =0.$$
Since  $\psi^{-1}_{\V^{\frac{1}{2}}\,\U\,\V^{\frac{1}{2}}}(z)=\chi_{\V^{\frac{1}{2}}\,\U\,\V^{\frac{1}{2}}}(z)$ and $\psi^{-1}_{\V}(z)=\chi_{\V}(z)$, from the above equations we get

$$\alpha z^2\chi_{\V^{\frac{1}{2}}\,\U\,\V^{\frac{1}{2}}}(z)
 -z (1+ (c -\alpha (1+\lambda))\chi_{\V^{\frac{1}{2}}\,\U\,\V^{\frac{1}{2}}}(z))-(c-\alpha\lambda)\chi_{\V^{\frac{1}{2}}\,\U\,\V^{\frac{1}{2}}}(z)=0,$$
and
$$z^2\alpha \chi_{\V}(z) +z(\alpha  \chi_{\V}(z) +\alpha\lambda  \chi_{\V}(z)-1)+\alpha\lambda  \chi_{\V}(z) =0.$$
Now we use \eqref{Str} to find the corresponding $S$-transform as
$$S_{\V^{\frac{1}{2}}\,\U\,\V^{\frac{1}{2}}}(z)=\frac{1}{\alpha\lambda-c +\alpha z},$$
and 
$$
S_{\V}(z)=\frac{1}{\alpha\lambda+\alpha z}.
$$
Note that the equations above define $S$-transform of the free-Poisson distribution (see equation \eqref{eq:StransformataPoison}).
Since $\U$ and $\V$ are free by \eqref{Scon} we arrive at
$$S_{\U}(z)=1+\frac{c}{\alpha\lambda+\alpha z-c}.$$

From \eqref{eq:StransormataBainomial} it follows that  $S_{\U}(z)$ is the Cauchy transform of free-binomial distribution with parameters $\theta=c/\alpha$ and $\sigma=\lambda-c/\alpha$.
\end{proof}

\subsection{Proof of Theorem \ref{tw:binomal}}
\begin{proof}
Similarly as the previous proof this proof is mainly based on Theorem \ref{twr:2}. \\
First note that if we define $\Psi_\X(z)=z\X(1-z\X)^{-1}$, then
\begin{align}
\label{eq:psi_tr2}
\frac{z}{z-1}\left(\Psi_\X(z)-\Psi_\X(1)\right)&=\frac{z}{z-1}
\left(\I-z\X\right)^{-1}\left(z\X(\I-\X)-(\I-z\X)\X\right)\left(\I-\X\right)^{-1}=
\left(\I-\X\right)^{-1}z\X(1-z\X)^{-1}\\
&=\left(\I-\X\right)^{-1}\Psi_\X(z).\nonumber
\end{align}

Using \eqref{eq:binomal2} we can write
\begin{align*}
\tau\left((\I-\Y)^{-1}\Psi_{\Y^{\frac{1}{2}}\X\Y^{\frac{1}{2}}}(z)\right)&=
\tau\left(\tau\left((\I-\Y)^{-1}|\Y^{\frac{1}{2}}\X\Y^{\frac{1}{2}}\right)\Psi_{\Y^{\frac{1}{2}}\X\Y^{\frac{1}{2}}}(z)\right)\\&=
d\tau\left(\left(\I-\Y^{\frac{1}{2}}\X\Y^{\frac{1}{2}}\right)^{-1}\Psi_{\Y^{\frac{1}{2}}\X\Y^{\frac{1}{2}}}(z)\right).
\end{align*}
By equation \eqref{eq:psi_tr2} we get
\begin{align}
\label{eq:binProof1}
\tau\left((\I-\Y)^{-1}\Psi_{\Y^{\frac{1}{2}}\X\Y^{\frac{1}{2}}}(z)\right)=
d\tau\left(\frac{z}{z-1}\left(\Psi_{\Y^{\frac{1}{2}}\X\Y^{\frac{1}{2}}}(z)-\Psi_{\Y^{\frac{1}{2}}\X\Y^{\frac{1}{2}}}(1)\right)\right).
\end{align}
On the other hand we can use Theorem \ref{twr:2} to transform \eqref{eq:binomal2}, which after applying \eqref{eq:psi_tr2} gives
\begin{align}
\label{eq:binProof2}
\tau\left((\I-\Y)^{-1}\Psi_{\Y^{\frac{1}{2}}\X\Y^{\frac{1}{2}}}(z)\right)=
\tau\left((\I-\Y)^{-1}\Psi_{\Y}\left(F(z)\right)\right)=
\frac{F(z)}{F(z)-1}\tau\left(\Psi_{\Y}\left(F(z)\right)-\Psi_{\Y}\left(1\right)\right),
\end{align}
where $\Psi_{\Y}\left(F(z)\right)=\Psi_{\Y^{\frac{1}{2}}\X\Y^{\frac{1}{2}}}(z)$.\\
Finally we can combine \eqref{eq:binProof1} and \eqref{eq:binProof2} which results
\begin{align} 
\label{eq:binProof3}
\frac{F(z)}{F(z)-1}\left(\psi _{\Y}\left(F(z)\right)-\psi_{\Y}\left(1\right)\right)=
d\frac{z}{z-1}\left(\psi_{\Y^{\frac{1}{2}}\X\Y^{\frac{1}{2}}}(z)-
\psi_{\Y^{\frac{1}{2}}\X\Y^{\frac{1}{2}}}(1)\right),
\end{align}
where for any non-commutative random variable $\X$ we denote $\psi_{\X}(z)=\tau\left(\Psi_{\X}(z)\right)$.
\\ Note that equation \eqref{eq:binomal2} implies
\begin{align*}
\tau\left(\left(\I-\Y\right)^{-1}\right)=d\tau\left(\left(\I-\Y^{\frac{1}{2}}\X\Y^{\frac{1}{2}}\right)^{-1}\right),
\end{align*}
the above equation can be rewritten as
\begin{align*}
1+\tau\left(\Y\left(\I-\Y\right)^{-1}\right)=d\left(1+\tau\left(\Y^{\frac{1}{2}}\X\Y^{\frac{1}{2}}\left(\I-\Y^{\frac{1}{2}}\X\Y^{\frac{1}{2}}\right)^{-1}\right)\right),
\end{align*}
which means
\begin{align*}
\psi_\Y(1)=d\left(1+\psi_{\Y^{\frac{1}{2}}\X\Y^{\frac{1}{2}}}(1)\right)-1.
\end{align*}
Putting the above equation into \eqref{eq:binProof3} gives us
\begin{align}
\label{eq:Bin_fin1}
\frac{F(z)}{F(z)-1}\left(\psi _{\Y}\left(F(z)\right)-d\left(1+\psi_{\Y^{\frac{1}{2}}\X\Y^{\frac{1}{2}}}(1)\right)+1\right)=
d\frac{z}{z-1}\left(\psi_{\Y^{\frac{1}{2}}\X\Y^{\frac{1}{2}}}(z)-
\psi_{\Y^{\frac{1}{2}}\X\Y^{\frac{1}{2}}}(1)\right).
\end{align}
Now we proceed similarly with equation \eqref{eq:binomal2}.
\begin{align*}
\tau\left((\I-\Y)\Psi_{\Y^{\frac{1}{2}}\X\Y^{\frac{1}{2}}}(z)\right)=
\tau\left(\tau\left((\I-\Y)|\Y^{\frac{1}{2}}\X\Y^{\frac{1}{2}}\right)\Psi_{\Y^{\frac{1}{2}}\X\Y^{\frac{1}{2}}}(z)\right)=
c\tau\left(\left(\I-\Y^{\frac{1}{2}}\X\Y^{\frac{1}{2}}\right)\Psi_{\Y^{\frac{1}{2}}\X\Y^{\frac{1}{2}}}(z)\right).
\end{align*}
By \eqref{eq:psi_prod} we have
\begin{align}
\label{eq:binProof4}
c\tau\left(\Psi_{\Y^{\frac{1}{2}}\X\Y^{\frac{1}{2}}}(z)-\frac{\Psi_{\Y^{\frac{1}{2}}\X\Y^{\frac{1}{2}}}(z)}{z}\right)
+c\tau\left(\Y^{\frac{1}{2}}\X\Y^{\frac{1}{2}}\right)=
c\tau\left(\frac{z-1}{z}\Psi_{\Y^{\frac{1}{2}}\X\Y^{\frac{1}{2}}}(z)\right)+c\tau\left(\Y^{\frac{1}{2}}\X\Y^{\frac{1}{2}}\right).
\end{align}
By Theorem \ref{twr:2}, by a similar computation we obtain 
\begin{align}
\label{eq:binProof5}
\tau\left((\I-\Y)\Psi_{\Y^{\frac{1}{2}}\X\Y^{\frac{1}{2}}}(z)\right)=
\tau\left((\I-\Y)\Psi_{\Y}(F(z))\right)=&
\tau\left(\frac{F(z)-1}{F(z)}\Psi_{\Y}(F(z))\right)+\tau\left(\Y\right)
\\ \nonumber=&\tau\left(\frac{F(z)-1}{F(z)}\Psi_{\Y^{\frac{1}{2}}\X\Y^{\frac{1}{2}}}(z)\right)+\tau\left(\Y\right).
\end{align}
Combining equations \eqref{eq:binProof4} and \eqref{eq:binProof5} we get
\begin{align*}
\tau\left(\frac{F(z)-1}{F(z)}\Psi_{\Y^{\frac{1}{2}}\X\Y^{\frac{1}{2}}}(z)\right)+\tau\left(\Y\right)=
c\tau\left(\frac{z-1}{z}\Psi_{\Y^{\frac{1}{2}}\X\Y^{\frac{1}{2}}}(z)\right)+\tau\left(\Y^{\frac{1}{2}}\X\Y^{\frac{1}{2}}\right).
\end{align*}
Taking into account that equation \eqref{eq:binomal1} implies $\tau\left(\I-\Y\right)=c\left(\I-\tau\left(\Y^{\frac{1}{2}}\X\Y^{\frac{1}{2}}\right)\right)$ we obtain
\begin{align}
\label{eq:Bin_fin2}
\frac{F(z)-1}{F(z)}\psi_{\Y^{\frac{1}{2}}\X\Y^{\frac{1}{2}}}(z)=c\frac{z-1}{z}\psi_{\Y^{\frac{1}{2}}\X\Y^{\frac{1}{2}}}(z)+c-1,
\end{align}
where we use the same notation as in the equation \eqref{eq:Bin_fin1}.

We can multiply equations \eqref{eq:Bin_fin1} and \eqref{eq:Bin_fin2} which gives us
\begin{align}
&\psi_{\Y^{\frac{1}{2}}\X\Y^{\frac{1}{2}}}(z)\left(\psi_{\Y^{\frac{1}{2}}\X\Y^{\frac{1}{2}}}(z)
-d\left(1+\psi_{\Y^{\frac{1}{2}}\X\Y^{\frac{1}{2}}}(1)\right)+1\right)=\\
&d\frac{z}{z-1}\left(\psi_{\Y^{\frac{1}{2}}\X\Y^{\frac{1}{2}}}(z)-
\psi_{\Y^{\frac{1}{2}}\X\Y^{\frac{1}{2}}}(1)\right)
\left(c\frac{z-1}{z}\psi_{\Y^{\frac{1}{2}}\X\Y^{\frac{1}{2}}}(z)+c-1\right),
\end{align}
we use here the relation $\psi_{\Y}\left(F(z)\right)=\psi_{\Y^{\frac{1}{2}}\X\Y^{\frac{1}{2}}}(z)$.\\
Now we define $\chi_{\Y^{\frac{1}{2}}\X\Y^{\frac{1}{2}}}=\psi_{\Y^{\frac{1}{2}}\X\Y^{\frac{1}{2}}}^{-1}$, then for $\chi_{\Y^{\frac{1}{2}}\X\Y^{\frac{1}{2}}}$ we obtain
\begin{align}
z\left(z
-d\left(1+\alpha\right)+1\right)=
d\frac{\chi_{\Y^{\frac{1}{2}}\X\Y^{\frac{1}{2}}}(z)}{\chi_{\Y^{\frac{1}{2}}\X\Y^{\frac{1}{2}}}(z)-1}\left(z-
\alpha\right)
\left(c\frac{\chi_{\Y^{\frac{1}{2}}\X\Y^{\frac{1}{2}}}(z)-1}{\chi_{\Y^{\frac{1}{2}}\X\Y^{\frac{1}{2}}}(z)}z+c-1\right),
\end{align}
where we denote $\alpha=\psi_{\Y^{\frac{1}{2}}\X\Y^{\frac{1}{2}}}(1)$.\\
From the above equation we get
\begin{align*}
\chi_{\Y^{\frac{1}{2}}\X\Y^{\frac{1}{2}}}(z)=\frac{z}{z+1}\left(1+\frac{1-d}{(c-1)d\alpha+z(1-cd)}\right).
\end{align*}
By equation \eqref{Str} and traciality of $\tau$ we get the $S$-transform of $\X\Y$
\begin{align}
\label{eq:SolutionXY}
S_{\X\Y}=1+\frac{1-d}{(c-1)d\alpha+z(1-cd)}.
\end{align}\\

We will obtain similarly the $S$-transform of $\Y$.\\
Now we rewrite \eqref{eq:Bin_fin1} and \eqref{eq:Bin_fin2} as
\begin{align*}
\frac{F(z)}{F(z)-1}\left(\psi _{\Y}\left(F(z)\right)-d\left(1+\alpha\right)+1\right)&=
d\frac{z}{z-1}\left(\psi _{\Y}\left(F(z)\right)-
\alpha\right)\\
\frac{F(z)-1}{F(z)}\psi _{\Y}\left(F(z)\right)-c+1&=c\frac{z-1}{z}\psi _{\Y}\left(F(z)\right).
\end{align*}
Multiplying both sides of the above equations gives us
\begin{align*}
\frac{F(z)}{F(z)-1}\left(\psi _{\Y}\left(F(z)\right)-d\left(1+\alpha\right)+1\right)
\left(\frac{F(z)-1}{F(z)}\psi _{\Y}\left(F(z)\right)-c+1\right)=\\
cd\left(\psi _{\Y}\left(F(z)\right)-
\alpha\right)\psi _{\Y}\left(F(z)\right).
\end{align*}
Substituting $z=F^{-1}(z)$ we obtain
\begin{align*}
\frac{z}{z-1}\left(\psi _{\Y}\left(z\right)-d\left(1+\alpha\right)+1\right)
\left(\frac{z-1}{z}\psi _{\Y}\left(z\right)-c+1\right)=\\
cd\left(\psi _{\Y}\left(z\right)-
\alpha\right)\psi _{\Y}\left(z\right),
\end{align*}
which for $\chi_\Y=\psi_\Y^{-1}$ implies
\begin{align*}
\frac{\chi_\Y(z)}{\chi_\Y(z)-1}\left(z-d\left(1+\alpha\right)+1\right)
\left(\frac{\chi_\Y(z)-1}{\chi_\Y(z)}z-c+1\right)=\\
cd\left(z-
\alpha\right)z.
\end{align*}
From the above equation we get
\begin{align*}
\chi_\Y(z)=\frac{z}{z+1}\left(1+\frac{c(1-d)}{(c-1)(d(1+\alpha)-1)+z(1-cd)}\right).
\end{align*}
By the equation \eqref{Str} the $S$-transform of $\Y$ has form
\begin{align}
S_\Y(z)=1+\frac{c(1-d)}{(c-1)(d(1+\alpha)-1)+z(1-cd)}.
\end{align}
From \eqref{eq:StransormataBainomial} we see that $\Y$ has a free binomial distribution with parameters $\sigma=\frac{(1-c) (d(\alpha+1)-1)}{cd-1}$ and $\theta=\frac{c(1-d)}{1-cd}$.

Since we know the $S$-transforms of $\X\Y$ and $\Y$ we can find the $S$-transform of $\X$ by the equation \eqref{Scon} which gives us
\begin{align*}
S_\X(z)=1+\frac{(c-1)(d-1)}{(c-1)d\alpha+z(1-cd)}.
\end{align*}
From \eqref{eq:StransormataBainomial} we see that $\X$ has a free binomial distribution with parameters $\sigma=\frac{(1-c)d\alpha}{cd-1}$ and $\theta=\frac{(c-1)(d-1)}{1-cd}$.
\end{proof}
\begin{center} Acknowledgments
\end{center} 
The authors would like to thank  M.\ Bo\.zejko  for several discussions and
helpful comments during the preparation of this paper. 
U.\ Franz also wants to thank Takahiro Hasebe for many stimulating
  discussions related to subordination and monotone independence. K.\ Szpojankowski thanks J.\ Weso\l{}owski for helpful discussions.

\end{document}